\documentclass[12pt]{amsart}

\usepackage{amsthm, amsmath, amssymb, amsrefs, fullpage}

\newcommand{\C}{\mathbb{C}} 
\newcommand{\D}{\mathbb{D}} 
\newcommand{\R}{\mathbb{R}}
\newcommand{\Z}{\mathbb{Z}} 
\newcommand{\N}{\mathbb{N}}
\newcommand{\al}{\alpha}

\newcommand{\tr}{\text{tr}}
\renewcommand{\Re}{\text{Re}}
\renewcommand{\Im}{\text{Im}}

\numberwithin{equation}{section}

\newtheorem{theorem}{Theorem}[section]

\newtheorem{lemma}[theorem]{Lemma}
\newtheorem{definition}[theorem]{Definition}

\title{Global bounds on stable polynomials}

\author{Greg Knese}
\address{Washington University in St. Louis\\ Department of
  Mathematics, One Brookings Drive, St. Louis, MO 63130}
\email{geknese@wustl.edu}

\thanks{This research was supported by NSF grant DMS-1363239}



\date{\today}

\begin{document}
\bibliographystyle{plain}

\maketitle

\begin{abstract}
A classical inequality of Sz\'asz bounds polynomials
with no zeros in the upper half plane entirely in terms of
their first few coefficients.   Borcea-Br\"and\'en generalized this
result to several variables as a piece of 
their characterization of linear maps
on polynomials preserving stability.  In this paper, we improve
Sz\'asz's original inequality, use determinantal representations to
prove Sz\'asz type inequalities in two variables, and then prove that
one can use the two variable inequality to prove an inequality for
several variables.  
\end{abstract}

\tableofcontents

\section{Introduction}

We say $p \in \C[z]$ is \emph{stable} if $p$ has
no zeros in $\C_+:=\{z\in \C: \Im z > 0\}$.

This note is about improvements and generalizations of the following
classical inequality of O.\ Sz\'asz.

\begin{theorem}[Sz\'asz \cite{Sz}] \label{szasz}
If $p(z) = \sum_{j=0}^{d}c_j z^j \in \C[z]$ is stable and $p(0)=1$ 
then
\[
|p(z)| \leq \exp(|z||c_1| + 3|z|^2(|c_1|^2 + |c_2|)).
\]
\end{theorem}

The purpose of the theorem is to prove 
\[
\mathcal{F}_C = \{p\in \C[z]: p \text{ is stable}, p(0)=1, |p'(0)|, |p''(0)| \leq C\}
\]
is a normal family whose local uniform limits 
are entire functions of order at most $2$.
One can use this to give a complete characterization of the local
uniform limits of stable polynomials.  See Theorem 4 Chapter VIII
(page 334) of \cite{Levin}.

Only recently have multivariable Sz\'asz type inequalities
been considered.  We say $p \in \C[z_1,\dots, z_n]$ 
is \emph{stable} if $p$ has no zeros in $(\C_{+})^n$.
In their groundbreaking characterization of linear operators 
on polynomials $T:\C[z_1,\dots, z_n] \to \C[z_1,\dots, z_n]$ that
preserve stability, Borcea-Br\"and\'en \cite{BB} established a Sz\'asz type
inequality. Its purpose was to 
prove that the \emph{symbol} of $T$
\[
\overline{G_T}(z,w) = \sum_{\alpha \in \N^n} (-1)^{\alpha} T(z^\alpha) \frac{w^\alpha}{\alpha!}
\]
is actually an entire function.  Formally, the symbol is 
given as $\overline{G_T}(z,w) = T[e^{-z\cdot w}]$.

We let $e_1,\dots, e_n$ be standard basis vectors
of $\Z^n$. 

\begin{theorem}[Borcea-Br\"and\'en Theorem 6.5\cite{BB}] \label{BBthm}
Suppose that $p(z) = \sum_\beta a(\beta) z^{\beta} \in \C[z_1,\dots, z_n]$ is stable with 
$p(0)=1$.  Let
\[
B = 2^{n-1} \frac{\sqrt{2e^2-e}}{e-1} = 2^{n-1}\cdot 2.0210\dots,
\]
\[
C=6e^2\left(\sum_{i=1}^{n}|a(e_i)|\right)^2 + 
4e^2\sum_{i,j=1}^{n} |a(e_i+e_j)|.
\]
Then,
\[
|p(z)| \leq B\exp(C\|z\|_{\infty}^2).
\]
\end{theorem}

Here $\|z\|_{\infty} = \max\{|z_j|: 1\leq j \leq n\}$.

The original proof of this theorem uses an
 inequality of Sz\'asz (actually inequality \eqref{BBsz} below)
along with some linear operators that preserve stability
to bound all of the coefficients $\{a(\beta)\}$ and
then reassemble $p$ to obtain the bound above.

The goal of this paper is to improve both Theorem \ref{szasz} and Theorem \ref{BBthm} and to give a more
geometric proof in the multivariable case.  

Our strategy is to first make some minor improvements 
to Sz\'asz's original argument, then to use determinantal
representations of two variable stable polynomials to prove 
a version of Theorem \ref{BBthm} in two variables, and finally
to show that an inequality in the two variable case yields an
inequality in several variables.  We also show how 
to handle the case where $p(0)=0$.  

First, Sz\'asz's inequality can be slightly improved to the following.

\begin{theorem} \label{szimprove}
Suppose $p(z) = \sum_{j=0}^{d}p_j z^j \in \C[z]$ is stable and $p(0)=1$.
Then,
\[
|p(z)| \leq \exp( \Re(p_1 z) + \frac{1}{2}(|p_1|^2 - 2\Re(p_2))|z|^2).
\]
The inequality is sharp on the imaginary axis for stable $p \in
\R[z]$. 
Specifically, given $c_1,c_2 \in \R$ with $\gamma := \frac{1}{2}(c_1^2-2c_2)> 0$
there exist stable polynomials $p_n \in \R[z]$ with $p_n=1+c_1z+c_2z^2
+ \dots$ such that
\[
\lim_{n\to \infty} |p_n(iy)| = \exp(\gamma y^2).
\]
\end{theorem}
 
We are unsure about sharpness more generally.
Note that necessarily $c_1^2-2c_2 \geq 0$ if $p=1+c_1z+c_2z^2+\dots
\in \R[z]$ is stable (e.g. examine $c_1,c_2$ in terms of roots)
and $c_1^2-2c_2=0$ if and only if $p\equiv 1$.

The next theorem subsumes this theorem, however 
we feel it is instructive to go through the 
elementary one variable argument as a warm-up in
Section \ref{sec:szimprove}.

Using determinantal representations for 
stable two variable polynomials we are able to offer 
the following improvement on Theorem \ref{BBthm}.
Below, $p_j = \frac{\partial p}{\partial z_j}$ and $p_{j,k} = \frac{\partial^2 p}{\partial z_j \partial z_k}$. 

\begin{theorem} \label{bisz}
Let $p\in \C[z_1,z_2]$ be stable.  If $p(0)=1$,
then 
\begin{equation} \label{mainineq}
|p(z)|\leq\exp(\Re(\sum_{j=1}^{2} z_j p_j(0)) +
\frac{1}{2}\|z\|_{\infty}^2 (|\sum_{j=1}^{2} p_j(0)|^2-\Re(\sum_{j,k=1}^{2} p_{j,k}(0)))).
\end{equation}
\end{theorem}

For comparison, using $a \leq (1+a^2)/2$ for the first term
in the exponential and writing $p = 1+\sum_{\beta\ne 0} a(\beta) z^{\beta}$ we get
\[
|p(z)| \leq \sqrt{e} \exp(C\|z\|_{\infty}^2)
\]
\[
C=(\sum_{j=1}^{2} |a(e_j)|)^2+\sum_{j,k=1}^{2} |\Re  [a(e_j+e_k)]|.
\]
Here we used $p_{jj}(0)= 2 a(2e_j)$ and $p_{jk}(0) = a(e_j+e_k)$ for
$j\ne k$.

It turns out that the two variable result can be used 
to prove an $n$-variable result.

\begin{theorem} \label{msz}
Suppose $p\in \C[z_1,\dots,z_n]$ is stable.  If $p(0)=1$ then 
\[
|p(z)| \leq \exp(
\sqrt{2}|\nabla p(0)||z|
+(|\nabla p(0)|^2+\|\Re Hp(0)\|)|z|^2)
\]
where $Hp$ is the Hessian of $p$ and $\|\cdot \|$ denotes operator norm.
\end{theorem}

We also get an inequality more comparable to Theorem \ref{BBthm}:
\[
|p(z)| \leq \sqrt{e}\exp(C \|z\|_{\infty}^2)
\]
\[
C = 2(\sum_{j=1}^{n} |a(e_j)|)^2
+2 \sum_{j,k=1}^{n} |\Re [a(e_j+e_k)]|).
\]

To close the introduction we discuss what happens when
$p(0)=0$.  In one variable there is only a minor issue because we can 
factor $p(z) = z^{k}q(z)= p_kz^k + p_{k+1}z^{k+1} +
p_{k+2}z^{k+2}+\dots$ and get a bound
depending on $p_k\ne 0 ,p_{k+1},p_{k+2}$:
\[
|p(z)| \leq |p_k| \exp \left[k (|z|-1) + \Re\left(\frac{p_{k+1}}{p_{k}}z\right) + 
\frac{1}{2}|z|^2\left( \left|\frac{p_{k+1}}{p_{k}}\right|^2 -
  2\Re\left(\frac{p_{k+2}}{p_k} \right) \right) \right]
\]
for stable $p \in \C[z]$ with a zero of order $k$ at $0$.
We used the inequality $\log |z|^k \leq k(|z|-1)$.  

In several variables the case $p(0)=0$ is a little more delicate.
Borcea-Br\"and\'en covered this case as follows.  
Given $p(z) =\sum_{\alpha} a(\alpha)z^{\alpha} \in \C[z_1,\dots,z_n]$
let  $\text{supp}(p) = \{\alpha \in \N^n: a(\alpha) \ne 0\}$ and let
$\mathcal{M}(p)$ denote the set of minimal elements of $\text{supp}(p)$
with respect to the partial order $\leq$ on $\N^n$.  Also, for fixed
$\mathcal{M} \subset \N^n$ let
\[
\mathcal{M}_2 = \{\alpha+ \beta : \alpha\in \mathcal{M}, \beta \in
\N^n, |\beta| \leq 2\}
\]

\begin{theorem}[Borcea-Br\"and\'en Theorem 6.6
  \cite{BB}] \label{BBthm2}
Let $\mathcal{M} \subset \N^n$ be a finite nonempty set and
$p(z) = \sum_{\alpha} a(\alpha) z^{\alpha} \in \C[z_1,\dots, z_n]$ be
stable
with $\mathcal{M}(p)= \mathcal{M}$. Then, there are constants $B$ and
$C$ depending only on the coefficients $a(\alpha)$ with $\alpha \in
\mathcal{M}_2$ such that 
\[
|p(z)| \leq B\exp(C\|z\|_{\infty}^2).
\]
Moreover, $B$ and $C$ can be chosen 
so that they depend continuously on the aforementioned set of
coefficients.
\end{theorem}

With our approach we are able to 
to get a more explicit estimate in two and several variables.
Set $\vec{1} = (1,\dots, 1) \in \C^n$.

\begin{theorem} \label{bisz2}
Let $p\in \C[z_1,z_2]$ be stable and assume $p$
vanishes to order $r$ at $0$.  Write out
the homogeneous expansion of $p$:
\[
p(z) = \sum_{j=r}^{d} P_j(z)
\]
where $P_j$ is homogeneous of degree $j$.  Then,
\[
|p(z)| \leq |P_r(\vec{1})|e^{-r/2}\exp\left[\Re(\sum_{j=1}^{2} c_j z_j)
+B \|z\|_{\infty}^2 \right]
\]
where
\[
c_j = \frac{1}{P_{r}(\vec{1})} \left[ \frac{\partial P_r}{\partial
      z_j}(\vec{1})\left(1-\frac{P_{r+1}(\vec{1})}{P_r(\vec{1})}\right)
 + \frac{\partial P_{r+1}}{\partial z_j}(\vec{1})\right]
\]
\[
B = \frac{1}{2}\left(\left|\frac{P_{r+1}(\vec{1})}{P_{r}(\vec{1})}\right|^2
  -2\Re\left(\frac{P_{r+2}(\vec{1})}{P_r(\vec{1})}\right) + r\right).
\]
\end{theorem}

This is proved in Section \ref{sec:bisz2}.  

Finally, we present a multivariable Sz\'asz inequality for 
the case $p(0)=0$.

\begin{theorem} \label{msz2}
Suppose $p\in \C[z_1,\dots,z_n]$ is stable and vanishes to order $r$ 
at $0$.  If we write out the homogeneous expansion of $p$
\[
p(z) = \sum_{j=r}^{d} P_j(z)
\]
then 
\[
|p(z)| \leq \|z\|_{\infty}^{r} |P_r(\vec{1})| \exp(C_0+ C_1
\|z\|_{\infty} + C_2 \|z\|_{\infty}^2)
\]
where $C_0, C_1,C_2$ are constants depending on $r, P_r(\vec{1}),
P_{r+1}(\vec{1}), P_{r+2}(\vec{1}), \nabla P_{r}(\vec{1}), \nabla
P_{r+1}(\vec{1})$.
\end{theorem}

The constants $C_0,C_1,C_2$ along with the proof of this theorem are
explicitly given in Section \ref{sec:msz2}.

\section{Improved one variable inequality}\label{sec:szimprove}
In this section we prove Theorem \ref{szimprove}.

\begin{lemma}\label{squareslemma}
Suppose $\al_1,\dots, \al_d\in \C$ with $\Im \al_j \leq 0$.  Then,
\[
\sum_{j=1}^{d} |\al_j|^2 \leq |\sum_{j=1}^{d} \al_j|^2 - 2 \Re \sum_{j<k} \al_j \al_k.
\]
Equality holds if and only if either $d=1$ or $\al_j \in \R$ for all $j$.
\end{lemma}

\begin{proof}
Note 
\[
|\sum_{j=1}^{d} \al_j|^2 = \sum_{j=1}^{d} |\al_j|^2 + 2 \Re\sum_{j<k} \al_j\overline{\al_k}
\]
So, our inequality reduces to showing the following is non-negative:
\[
2\Re \sum_{j<k} (\al_j \overline{\al_k} - \al_j \al_k) = 2\Re \sum_{j<k} \al_j(-2i)\Im\al_k
= 4\sum_{j<k} \Im \al_j \Im \al_k.
\]
The last quantity is evidently non-negative 
and equals zero exactly when either $d=1$ (and the sum is empty) or $\Im \al_j=0$ for
all $j$.
\end{proof}

Sz\'asz uses the inequality $|(1+z)e^{-z}|\leq e^{|z|^2}$ instead
of the stronger inequality:

\begin{lemma} \label{loglemma}
For $z \in \C$, $z\ne -1$
\[
\log|1+z|\leq \Re z + \frac{1}{2}|z|^2.
\]
\end{lemma}
\begin{proof}
Since $\log (1+x) \leq x$ we have
\[
\begin{aligned}
\log|1+z|&=\frac{1}{2}\log|1+z|^2 \\
&= \frac{1}{2}(\log(1+2\Re z + |z|^2)\\
&\leq \frac{1}{2}(2\Re z + |z|^2).
\end{aligned}
\]
\end{proof}

\begin{proof}[Proof of Theorem \ref{szimprove}]
Write $p(z) = \prod_{j=1}^{d}(1+\al_j z)$ where $\Im \al_j \leq 0$.  
Note $\sum_{j} \al_j = p_1$ and $\sum_{j<k} \al_j \al_k = p_2$.
By Lemma \ref{loglemma} and Lemma \ref{squareslemma}
\[
\begin{aligned}
\log |p(z)| & \leq \sum_{j} (\Re(\al_j z) + \frac{1}{2}|\al_j|^2|z|^2) \\
& = \Re(p_1 z) + \frac{1}{2}(\sum_{j} |\al_j|^2) |z|^2 \\
& \leq \Re(p_1 z) + \frac{1}{2}(|p_1|^2 - 2\Re p_2)|z|^2.
\end{aligned}
\]

Regarding sharpness define $\gamma = (c_1^2-2c_2)/2 >0$.   Choose $n$
large enough that $d_n = \gamma - c_1^2/(2n)\geq 0.$  Then,
the polynomial
\[
p_n(z) = \left(1+\frac{c_1z}{n}\right)^n\left(1-\frac{d_n z^2}{n}\right)^n
\]
is stable, belongs to $\R[z]$, and has the correct normalizations.
Since $p_n(z) \to \exp{(c_1z-\gamma z^2)}$ locally uniformly, we have
\[
\lim_{n\to \infty} |p_n(iy)| = \exp{\gamma y^2}
\]
which is exactly what was claimed.
\end{proof}

It is worth pointing out that Sz\'asz proves
\[
\sum_j |\alpha_j|^2 \leq 2|\sum_j \alpha_j|^2+|\sum_j \alpha_j^2|
\]
and then converts $\sum_j \alpha_j^2 = p_1^2-2p_2$ to get
the estimate
\begin{equation}\label{BBsz}
\sum_j |\alpha_j|^2 \leq 3|p_1|^2+2|p_2|.
\end{equation}
By sidestepping the former inequality and estimating 
$\sum_j |\al_j|^2$ directly in terms of polynomial 
coefficients we get a better bound.  
The inequality \eqref{BBsz} is used in \cite{BB} to
prove multivariable Sz\'asz inequalities.  
So, using Lemma \ref{squareslemma} in their proof would already
improve Theorem \ref{BBthm}.


%
%

\section{Two variable Sz\'asz inequality}

Using determinantal formulas it is possible to 
establish a Sz\'asz inequality for two variable polynomials.

\begin{definition} \label{detdef}
We shall say a stable polynomial $p\in \C[z_1,\dots,z_n]$ of
total degree $d$ 
has a \emph{determinantal representation} if
there exist $d\times d$ matrices $A,B_1,\dots, B_n$ and a constant $c\in \C$ such that
\begin{enumerate}
\item $\Im A := \frac{1}{2i}(A-A^*) \geq 0$
\item for all $j$, $B_j \geq 0$
\item $\sum_{j=1}^{n} B_j=I$.
\item $p(z) = c \det(A+\sum_{j=1}^{n} z_j B_j)$.
\end{enumerate}
\end{definition}

Theorem \ref{bisz} will be broken into two theorems.

\begin{theorem} \label{detrep}
If $p\in \C[z_1,z_2]$ is stable, then $p$ has a determinantal representation.
\end{theorem}

Several different determinantal representations are closely related to
this one but not quite equivalent.  There are determinantal
representations for three variable hyperbolic polynomials, two
variable real-zero
polynomials, and two variable real-stable polynomials (see
\cites{HV, VV, Grinshpan}).  It turns out this formula can
be derived from a determinantal representation for 
polynomials with no zeros on the bidisk $\D^2 = \{(z_1,z_2):
|z_1|,|z_2|<1\}$ from \cite{Grinshpan}.  
We show how to convert from the bidisk formula to Theorem \ref{detrep}
in Section \ref{sec:detrep}.
The method of conversion is a very slight modification of what
is done in the paper \cite{GK}.  We include the argument 
for the reader's convenience; the essence of Section \ref{sec:detrep}
is not new.

In Section \ref{sec:ineqfordets} we prove the following Sz\'asz inequality
for stable polynomials with determinantal representations.  

\begin{theorem} \label{ineqfordets}
Suppose $p\in \C[z_1,\dots, z_n]$ has
a determinantal representation as above.  
If $p(0)=1$, then 
\[
|p(z)|\leq\exp(\Re(\sum_{j=1}^{n} z_j p_j(0)) +
\frac{1}{2}\|z\|_{\infty}^2 (|\sum_{j=1}^{n} p_j(0)|^2-\Re(\sum_{j,k=1}^{n} p_{j,k}(0)))).
\]
\end{theorem}

Theorems \ref{detrep} and \ref{ineqfordets} combine to give Theorem \ref{bisz}.

\section{Determinantal representations} \label{sec:detrep}
In this section we prove Theorem \ref{detrep}.
We begin by recalling the following.

\begin{theorem}[See \cite{Grinshpan} Theorem 2.1] \label{bidiskdetrep} If $q\in \C[z_1,z_2]$ has no zeros in 
$\D^2$ and bidegree $(n,m)$, then there exists a constant $c$ and an $(n+m)\times(n+m)$ contractive matrix $D$ 
such that
\begin{equation} \label{qdet}
q(z) = c \det(I-D \Delta(z))
\end{equation}
where $\Delta(z) = z_1P_1+z_2P_2$ and 
\[
P_1 = \begin{pmatrix} I_n & 0 \\ 0 & O_m \end{pmatrix} \qquad 
P_2 = \begin{pmatrix} O_n & 0 \\ 0 & I_m \end{pmatrix}.
\]
\end{theorem}

Let $p \in \C[z_1,z_2]$ be stable 
and have bidegree $(n,m)$.  Define $\phi(\zeta) = i\frac{1+\zeta}{1-\zeta}$ and
\[
q(z_1,z_2) = p(\phi(z_1),\phi(z_2))\left(\frac{1-z_1}{2i}\right)^n \left(\frac{1-z_2}{2i}\right)^m.
\]
One can calculate that $\phi^{-1}(\zeta) = \frac{\zeta - i}{\zeta +
  i}$ and
\[
p(z_1,z_2) = q(\phi^{-1}(z_1),\phi^{-1}(z_2)) (z_1+i)^n(z_2+i)^m.
\]
Then, $q$ has no zeros in $\D^2$ and so the conclusion
of Theorem \ref{bidiskdetrep} holds.
Then, converting \eqref{qdet} to a formula for $p$ yields
\[
\begin{aligned}
p(z) &= c\det( (z_1+i)P_1 + (z_2+i)P_2 - D((z_1-i)P_1+(z_2-i)P_2))\\
&= c \det( (I-D)\Delta(z) +i (I+D))
\end{aligned}
\]
Since $D$ is a contraction, the eigenspace corresponding to eigenvalue $1$ is reducing (if nontrivial). 
Thus, there exists a unitary $U$ such that
\[
D = U\begin{pmatrix} I & 0 \\ 0 & K \end{pmatrix} U^*
\]
where $K$ is a contractive $k\times k$ matrix for which $1$ is not an 
eigenvalue. Here $k$ is the codimension of the eigenspace of $D$ corresponding to eigenvalue $1$.
 
Then,
\[
\begin{aligned}
p(z) &= c \det\left( \begin{pmatrix} 0 & 0 \\ 0 & I-K \end{pmatrix} U^* \Delta(z) U + 
i\begin{pmatrix} 2I & 0 \\ 0 & I+K \end{pmatrix}\right)\\
&= c \det(I-K)
\det\left( \begin{pmatrix} 0 & 0 \\ 0 & I \end{pmatrix} U^*\Delta(z) U + \begin{pmatrix} 2iI & 0 \\ 0 & A \end{pmatrix}\right)
\end{aligned}
\]
where $A = i(I+K)(I-K)^{-1}$. 

Let $B_j$ equal the bottom right $k\times k$ block of $U^*P_jU$.  Then,
\[
p(z) = c\det(I-K) 
\det\begin{pmatrix} 2iI & 0 \\ * & A + \sum_j z_jB_j \end{pmatrix}
= c_0 \det(A+\sum_j z_j B_j)
\]
where $c_0$ is a new constant (the $*$ denotes a block we are
unconcerned with).  Since $P_1+P_2 = I$, 
$B_1+B_2= I$.  Also note that $p(t,t) = c_0 \det(A+tI)$ has
degree $k$ so that $k \leq \deg p$.  On the other hand, 
the determinantal formula for $p$ has total degree at most $k$, so that $\deg p \leq k$.  
Therefore the matrices in our formula have size matching
the total degree of $p$.
Finally, 
\[
\Im A 
 = (I-K)^{-1}(I-KK^*)(I-K^*)^{-1} \geq 0.
\]
This proves Theorem \ref{detrep}.

\section{Sz\'asz inequality for determinantal polynomials}\label{sec:ineqfordets}
In this section we prove Theorem \ref{ineqfordets}.

Suppose $p(z) = c \det(A+\sum_{j=1}^{n} z_j B_j)$ where
$\sum_{j=1}^{n} B_j = I$, $B_j \geq 0$, $\Im A \geq 0$, and 
$p(0)=1$.  By the last normalization $A$ is invertible
with $c\det A = 1$ so that
\[
p(z) = \det( I + \sum_{j=1}^{n} z_j X_j)
\]
where $X_j = B_j A^{-1}$.  As with complex numbers, $\Im (A^{-1}) \leq 0$.

It helps to make note of a few formulas for the 
derivatives of $p$. 
Recall that if $A(t)$ is a differentiable matrix function then
\[
\frac{d}{dt} \det A(t) = \tr( A'(t) A(t)^{-1})) \det A(t)
\]
whenever $A(t)$ is invertible. Here $\tr$ is the trace of a matrix.

Letting $X(z) = I+\sum_{j=1}^n z_j X_j$, whenever $p(z) \ne 0$
we have
\[
\begin{aligned}
p_j(z) &= \text{tr}(X_j(X(z))^{-1}) p(z)\\
p_{jk}(z) &= -\text{tr}(X_j (X(z))^{-1}X_k(X(z))^{-1})p(z) +
\text{tr}(X_j(X(z))^{-1})\text{tr}(X_k(X(z))^{-1})p(z)
\end{aligned}
\]
so that
\[
p_j(0) = \text{tr} X_j \quad p_{jk}(0)=-\text{tr}(X_jX_k) + \text{tr}(X_j)\text{tr}(X_k).
\]

For a positive definite matrix $P$ we have $\log P \leq P-I$ 
simply because the same inequality holds for the eigenvalues of $P$.
Therefore, 

\begin{align}
\log|p(z)| &= \frac{1}{2} \log \det(X(z)^*X(z)) \label{logp1} \\
&= (1/2) \tr \log X(z)^*X(z) \label{logp2} \\
& \leq (1/2) \tr (X(z)^*X(z) - I) \label{logp3} \\
& = (1/2) \tr (2 \Re(\sum_{j=1}^{n} z_j X_j)+ (\sum_{j=1}^{n} z_j X_j)^*(\sum_{k=1}^{n} z_k X_k)). \nonumber
\end{align}

Now, 
\[
\begin{aligned}
\tr (\sum_{j}z_j X_j)^*(\sum_{k} z_k X_k) &= \tr [(A^*)^{-1}
(\sum_{j}z_j B_j)^*(\sum_{k} z_k B_k) A^{-1}]\\
& = \tr[ (\sum_{j}z_j
B_j)^*(\sum_{k} z_k B_k) A^{-1}(A^*)^{-1}].
\end{aligned}
\]

By Lemma \ref{PM}, Lemma \ref{sumB}, and Lemma \ref{Im} below we have
\[
\begin{aligned}
\tr (\sum_{j}z_j X_j)^*(\sum_{k} z_k X_k) &\leq \|\sum_j z_j B_j\|^{2}
\tr \left[(A^*)^{-1} A^{-1}\right] \\
&\leq \|z\|_{\infty}^2\left[ |\tr(A^{-1})|^2 - \Re(
(\tr A^{-1})^2 - \tr A^{-2})\right].
\end{aligned}
\]


Finally, since $\sum_j B_j = I$ we have
\[
\sum_j p_j(0) = \tr A^{-1}
\]
and 
\[
\sum_{j,k} p_{jk}(0) = (\tr A^{-1})^2 - \tr A^{-2}.
\]
Thus,
\[
\log|p(z)| \leq \Re(\sum_{j=1}^{n} p_j(0)z_j) + \frac{1}{2} \|z\|_{\infty}^2 (|\sum_{j=1}^{n} p_j(0)|^2 - \Re(\sum_{j,k=1}^{n} p_{jk}(0)))
\]
which proves Theorem \ref{ineqfordets} modulo the following three lemmas.

\begin{lemma} \label{PM} Let $P, M$ be $n\times n$ matrices.
If $P \geq 0$, then
\[
|\text{tr}(MP)| \leq \|M\| \text{tr}P
\]
\end{lemma}
\begin{proof}
Since $P\geq 0$, we can decompose $P = \sum_j v_jv_j^*$ where
$v_j \in \C^n$.  Then,
\[
|\text{tr}{MP}| \leq \sum_j |\text{tr}{Mv_jv_j^*}|
= \sum_j |\langle Mv_j, v_j \rangle|
\leq \|M\| \sum_j \|v_j\|^2 = \|M\| \text{tr}{P}.
\]
\end{proof}

The following is a standard result.

\begin{lemma} \label{sumB}
Suppose $B_1,\dots, B_n$ are $N\times N$ matrices.  
Assume for all $j$, $B_j \geq 0$ and $\sum_j B_j = I$.
Then, there exist pairwise orthogonal projection matrices $P_1,\dots, P_n$
of size $m\times m$ where $m= nN$ such that
\begin{equation} \label{bj}
 B_j = (I_N,0,\dots, 0) P_j (I_N,0,\dots, 0)^t.
\end{equation}
In particular, for $z=(z_1,\dots, z_n) \in \C^n$
\[
\|\sum_{j} z_j B_j\| \leq \|z\|_{\infty}.
\]
\end{lemma}

\begin{proof}
We can factor $B_j = A_j^* A_j$ with $N\times N$ matrix $A_j$.  
The $nN\times N$ matrix
\[
T = \begin{pmatrix} A_1 \\ \vdots \\ A_n\end{pmatrix}
\]
is an isometry from $\C^N$ to $\C^{nN}$ since $T^*T = \sum_j B_j = I$.
We can extend $T$ to a $m\times m$ unitary $U$.  Let
$Q_j$ be the orthogonal projection onto the $j$-th block
of $\C^{m} = \C^N\oplus \cdots \oplus \C^N$.  Set $P_j = U^*Q_j U$.
Then, \eqref{bj} holds and
\[
\|\sum_{j} z_j B_j\| \leq \|\sum_j z_j P_j \| \leq \|z\|_{\infty}.
\]
\end{proof}

The following lemma is an adaptation of our
one variable argument.

\begin{lemma} \label{Im}
If $M$ is a square matrix with $\Im M \geq 0$, then
\[
\text{tr}M^*M \leq |\text{tr}M|^2 - \Re( (\text{tr} M)^2 - \text{tr} M^2).
\]
\end{lemma}
\begin{proof}
Write $A = \Re M, B = \Im M$.  Then,
\[
\begin{aligned}
|\tr M|^2 &= (\tr A)^2+(\tr B)^2,\\
\tr M^*M &= \tr(A^2+B^2 + i(AB-BA)) = \tr A^2 + \tr B^2,\\
\Re (\tr M)^2 &= (\tr A)^2 - (\tr B)^2,\\
\Re [\tr M^2] &= \Re [\tr( A^2-B^2 + i(AB+BA))] = \tr A^2-\tr B^2. 
\end{aligned}
\]
Then,
\[
|\text{tr}M|^2-\tr M^*M - \Re( (\text{tr} M)^2 - \text{tr} M^2) = 2((\tr B)^2 - \tr B^2)
\]
If $B$ has eigenvalues $\beta_j\geq 0$ then
\[
(\tr B)^2 - \tr B^2 = (\sum \beta_j)^2 - \sum \beta_j^2 
= \sum_{j\ne k} \beta_j\beta_k \geq 0.
\]
This proves the claimed inequality.
\end{proof}

\section{Sz\'asz inequality for determinants with $p(0)=0$} \label{sec:bisz2}

As with Theorem \ref{bisz} we will prove a Sz\'asz inequality
for polynomials with determinantal representations and Theorem
\ref{bisz2}
will follow via Theorem \ref{detrep}.

\begin{theorem} \label{szdetreps2}
Suppose $p\in \C[z_1,\dots,z_n]$ has a
 determinantal representation
as in Definition \ref{detdef}.
Assume $p$
vanishes to order $r$ at $0$.  Write out
the homogeneous expansion of $p$:
\[
p(z) = \sum_{j=r}^{d} P_j(z)
\]
where $P_j$ is homogeneous of degree $j$.  Then,
\[
|p(z)| \leq |P_r(\vec{1})|e^{-r/2}\exp\left[\Re(\sum_{j=1}^{n} c_j z_j)
+B \|z\|_{\infty}^2 \right]
\]
where
\[
c_j = \frac{1}{P_{r}(\vec{1})} \left[ \frac{\partial P_r}{\partial
      z_j}(\vec{1})(1-\frac{P_{r+1}(\vec{1})}{P_r(\vec{1})})
 + \frac{\partial P_{r+1}}{\partial z_j}(\vec{1})\right]
\]
\[
B = \frac{1}{2}\left(\left|\frac{P_{r+1}(\vec{1})}{P_{r}(\vec{1})}\right|^2
  -2\Re\left(\frac{P_{r+2}(\vec{1})}{P_r(\vec{1})}\right) + r\right)
\]
and $\vec{1} = (1,\dots,1) \in \C^n$.
\end{theorem}

\begin{proof}
Write $p(z) = c\det(A+\sum_{j=1}^{d} z_j B_j)$ as 
in Definition \ref{detdef}.  Since $p(0)=0$, $\det A = 0$.
Since $\Im A \geq 0$, the eigenspace
corresponding to eigenvalue $0$ is reducing for $A$; see
Lemma \ref{ImA} below.  Let $s$ equal the dimension of
the kernel of $A$.

So, after conjugating by a unitary we can rewrite $p$ in the form
\[
p(z) = c\det\left( \begin{pmatrix} 0 & 0 \\ 0 & C \end{pmatrix} +
\sum_{j=1}^{d} z_j B_j \right)
\]
where $C$ is an invertible $(d-s)\times(d-s)$ matrix with $\Im C\geq 0$ and
the $B_j$ are relabelled after conjugating (they satisfy all of the
same properties as before).  Define $X_j = B_j \begin{pmatrix} I & 0
  \\ 0 & C^{-1} \end{pmatrix}$ and $J = \begin{pmatrix} O_s & 0 \\ 0 &
  I_{d-s} \end{pmatrix}$.  Then,
\[
p(z) = c_0 \det (J +\sum_{j=1}^{d} z_j X_j)
\]
where $c_0 = c \det C$.  Let $X(z) = J +\sum_{j=1}^{d} z_j X_j$.  Let
$X_s(z)$ be the top left $s\times s$ block of $X(z)$. Evaluating $\det
X$
starting with the top left $s\times s$ block gives
\[
\det X(z) = \det X_s(z) + \text{ higher order terms}.
\]
Note that $\det X_s(z)$ is homogeneous of degree $s$ and 
$X_s(\vec{1}) = I_s$ since $\sum_j B_j=I$.  This proves $s=r$.

We can follow some of the argument in Section \ref{sec:ineqfordets}.
Equations \eqref{logp1}, \eqref{logp2}, \eqref{logp3} hold when
$p(z)\ne 0$ but \eqref{logp3} rearranges into
\[
\begin{aligned}
\log|p(z)/c_0| & \leq (1/2) \text{tr}\left( \begin{pmatrix} -I_r & 0 \\ 0 &
  0 \end{pmatrix} + 2\Re(\sum_{j=1}^{n} z_j X_j) + (\sum_j z_j
X_j)^*(\sum_k z_k X_k)\right)\\
& \leq -r/2 + \Re(\sum_{j=1}^{n} z_j \tr(X_j)) + (1/2)\tr (\sum_{j}z_j
X_j)^*(\sum_k z_k X_k).
\end{aligned}
\]

As before using Lemmas \ref{PM}, \ref{sumB}, \ref{Im} we have
\[
\begin{aligned}
\tr (\sum_{j}z_j
X_j)^*(\sum_k z_k X_k) &\leq  \|\sum_{j} z_j B_j\|^2 \tr( \begin{pmatrix} I_r & 0 \\ 0 &
  (C^*)^{-1}C^{-1} \end{pmatrix}) \\
&\leq \|z\|_{\infty}^2( r + |\tr(C^{-1})|^2 - \Re(
(\tr C^{-1})^2 - \tr C^{-2})).
\end{aligned}
\]

Now we must relate these quantities 
to intrinsic quantities of $p$.  

First, $p(t\vec{1}) = c_0 \det \begin{pmatrix} tI_r & 0 \\ 0 &
  I+tC^{-1} \end{pmatrix} = c_0 t^r \det(I+tC^{-1}).$ 
So, using this formula and the homogeneous expansion of $p$ we get
\[
\begin{aligned}
 t^{-r}p(t\vec{1}) \big|_{t=0} &=  c_0  = P_r(\vec{1}) \\ 
 \frac{d}{dt}\big|_{t=0} t^{-r}p(t\vec{1})  &= c_0 \tr C^{-1} = P_{r+1}(\vec{1})\\
 \frac{d^2}{dt^2}\big|_{t=0} t^{-r} p(t\vec{1}) & = c_0( (\tr C^{-1})^2 -
 \tr C^{-2}) = 2P_{r+2}(\vec{1}).
\end{aligned}
\]
It is more difficult to calculate $\tr X_j$.
Define
\[
q(s,t) = p(s e_j + t\vec{1}) = c_0 \det\left(J + s X_j + t \begin{pmatrix} I & 0 \\
  0 & C^{-1} \end{pmatrix}\right). 
\]
Note $\frac{\partial q}{\partial s}(0,t) = \frac{\partial p}{\partial z_j}(t\vec{1})$.
Then, 
\[
\begin{aligned}
\frac{\partial q}{\partial s}(0,t) &= \tr \left( X_j \begin{pmatrix}
t^{-1}I_r & 0 \\ 0 & (I+tC^{-1})^{-1} \end{pmatrix}\right)  p(t\vec{1}) \\
&= \tr\left(X_j \begin{pmatrix}
I_r & 0 \\ 0 & t(I+tC^{-1})^{-1} \end{pmatrix}\right) c_0 t^{r-1}
\det(I+tC^{-1}) 
\end{aligned}
\]
Thus, we can do the following computation with matrices and
also with the homogeneous expansion of $p$ 
\[
t^{-r+1}\frac{\partial q}{\partial s}(0,t)\Big|_{t=0} = c_0 \tr
\left(X_j \begin{pmatrix} I_r & 0 \\ 0 & 0 \end{pmatrix}\right) = \frac{\partial
  P_r}{\partial z_j}(\vec{1})
\]
Therefore,
\[
\frac{\partial q}{\partial s}(0,t) -
t^{-1}\frac{p(t\vec{1})}{P_r(\vec{1})} \frac{\partial P_r}{\partial
  z_j}(\vec{1}) =
\tr\left( X_j \begin{pmatrix} 0 & 0 \\ 0 &
  t(I+tC^{-1})^{-1} \end{pmatrix}\right)c_0 t^{r-1} \det(I+tC^{-1})
\]
which implies
\[
t^{-r}\left(\frac{\partial q}{\partial s}(0,t) -
t^{-1}\frac{p(t\vec{1})}{P_r(\vec{1})} \frac{\partial P_r}{\partial
  z_j}(\vec{1})\right)\Big|_{t=0} = \tr\left(X_j \begin{pmatrix} 0 & 0 \\ 0 &
  I\end{pmatrix}\right) c_0 =\frac{\partial P_{r+1}}{\partial z_j}(\vec{1})
- \frac{P_{r+1}(\vec{1})}{P_r(\vec{1})} \frac{\partial P_r}{\partial z_j}(\vec{1})
\]
Therefore,
\[
c_0 \tr X_j = \left(1-\frac{P_{r+1}(\vec{1})}{P_r(\vec{1})}\right)
\frac{\partial P_r}{\partial z_j}(\vec{1}) + \frac{\partial
  P_{r+1}}{\partial z_j}(\vec{1}).
\]
If we reassemble we get
\[
\log|p(z)/P_r(\vec{1})| \leq -r/2 + \Re(\sum_{j=1}^{n} c_j z_j ) +
B \|z\|_\infty^2
\]
where
\[
c_j = \frac{1}{P_r(\vec{1})} \left[\left(1-\frac{P_{r+1}(\vec{1})}{P_r(\vec{1})}\right)
\frac{\partial P_r}{\partial z_j}(\vec{1}) + \frac{\partial
  P_{r+1}}{\partial z_j}(\vec{1})\right]
\]
\[
B = \frac{1}{2}\left( r+ \left|
    \frac{P_{r+1}(\vec{1})}{P_r(\vec{1})}\right|^2
  -2\Re\left(\frac{P_{r+2}(\vec{1})}{P_r(\vec{1})}\right)\right)
\]
and this concludes the proof.
\end{proof}

\begin{lemma} \label{ImA}
Suppose $A$ is a matrix with $\Im A \geq 0$.  
If $0$ is an eigenvalue of $A$ with eigenspace of dimension $s$, 
then there exists a unitary $U$ such that
\[
U^* A U = \begin{pmatrix} O_{s} & 0 \\ 0 & C\end{pmatrix}
\]
where $\Im C \geq 0$ and $C$ is invertible.
\end{lemma}

\begin{proof}
If we write $A$ using an orthonormal basis for
its kernel followed by an orthonormal basis for the orthogonal
complement of its kernel, we can put $A$ into the form
\[
\begin{pmatrix} O_s & B \\ 0 & C\end{pmatrix}
\]
by conjugating by a unitary.  This matrix will still have positive
semi-definite
imaginary part:
\[
\begin{pmatrix} O_s & \frac{1}{2i} B \\ -\frac{1}{2i} B^* & \Im
  C\end{pmatrix} \geq 0.
\]
This implies $B=0$.
Note $C$ is invertible because it cannot have $0$ as an eigenvalue.
\end{proof}

\section{Multivariable Sz\'asz inequalities}
Using the two variable Sz\'asz inequality 
we can establish the multivariable inequality Theorem \ref{msz}.

We will frequently use the component-wise 
partial order on $\R^n$: $x \geq y$
if and only if for all $j=1,\dots, n$, $x_j\geq y_j$.

\begin{proof}[Proof of Theorem \ref{msz}]
For $z=x+iy \in \C_+^{n}$ we have
\[
|p(x_1+iy_1,\dots, x_n+iy_n)| \geq |p(x_1\pm i y_1, \dots, x_n\pm iy_n)|
\]
for all independent choices of $\pm$ by Lemma \ref{Polya2}
below (more precisely, we can hold fixed any variables with a ``$+$'' and apply Lemma
\ref{Polya2} to the remaining variables).  So, it is enough 
to prove Theorem \ref{msz} for $z \in \C_+^n$.  

By Lemma \ref{Polya3} below, if $0\leq y\leq \tilde{y}$ then
\[
|p(x+iy)| \leq |p(x+i\tilde{y})|.
\]
Define $\tilde{y}$ as the vector with $j$-th component
\[
\tilde{y}_j = \max(|x_j|, y_j).
\]
Then, $\tilde{y} \geq \pm x$ and $\tilde{y}\geq y$.

Define
\[
q(w_1,w_2) = p(w_1(\tilde{y}+x)+w_2(\tilde{y}-x))
\]
which has no zeros in $\C_+^2$ and $q(0)=1$.
We will now apply Theorem \ref{bisz} using all of the 
following computations.
\[
q\left(\frac{i+1}{2}, \frac{i-1}{2}\right) = p(x+i\tilde{y})
\]
\[
q_1(w) = \sum_j p_j(w_1(\tilde{y}+x)+w_2(\tilde{y}-x))(\tilde{y}_j+x_j)
\]
\[
q_2(w) = \sum_j p_j(w_1(\tilde{y}+x)+w_2(\tilde{y}-x))(\tilde{y}_j-x_j)
\]
\[
q_{11}(0) = \sum_{j,k} p_{jk}(0)(\tilde{y}_j+x_j)(\tilde{y}_k+x_k)
\]
\[
q_{12}(0) = \sum_{j,k} p_{jk}(0)(\tilde{y}_j+x_j)(\tilde{y}_k-x_k)
\]
\[
q_{22}(0) = \sum_{j,k} p_{jk}(0)(\tilde{y}_j-x_j)(\tilde{y}_k-x_k)
\]
\[
q_1(0)(i+1)/2 + q_2(0)(i-1)/2 = \nabla p(0)\cdot (x+i\tilde{y})
\]
\[
q_1(0)+q_2(0) = 2\nabla p(0)\cdot \tilde{y}
\]
\[
q_{11}(0)+2q_{12}(0)+q_{22}(0) = 4\sum_{j,k}p_{jk}(0)\tilde{y}_j\tilde{y}_k.
\]

Thus, by Theorem \ref{bisz}

\begin{multline}\label{worth}
\log\left|q\left(\frac{i+1}{2},\frac{i-1}{2}\right)\right|  \leq
\Re(\nabla p(0)\cdot (x+i\tilde{y})) \\
+\frac{1}{4}(|2\nabla p(0)\cdot \tilde{y}|^2 
-4Re(\sum_{j,k} p_{jk}(0)\tilde{y}_j\tilde{y}_k)) \\
\leq 
\sqrt{2}|\nabla p(0)||z|
+(|\nabla p(0)|^2+\|\Re Hp(0)\|)|z|^2 
\end{multline}
where we have used $|x+i\tilde{y}| \leq \sqrt{2}|z|$ and $|\tilde{y}|\leq |z|$.  
\end{proof}

Since Theorem \ref{BBthm} is an estimate on polydisks it is worth
pointing out that \eqref{worth} yields
\[
\begin{aligned}
\log|p(z)| &\leq \|z\|_{\infty} \sqrt{2} \sum_j |p_j(0)| + \|z\|_{\infty}^2((\sum_j |p_j(0)|)^2
+ \sum |\Re [p_{jk}(0)]|) \\
& \leq \|z\|_{\infty} \sqrt{2} \sum_j |a(e_j)| + \|z\|_{\infty}^2((\sum_j |a(e_j)|)^2
+2 \sum |\Re [a(e_j+e_k)]|) \\
&\leq \frac{1}{2} + \|z\|_{\infty}^2(2(\sum_j |a(e_j)|)^2
+2 \sum |\Re [a(e_j+e_k)]|)
\end{aligned}
\]
where $p = \sum a(\beta) z^{\beta}$ and in the last line we used the
inequality $a\leq (1+a^2)/2$.
This gives
\[
|p(z)| \leq \sqrt{e}\cdot \exp(C\|z\|_{\infty}^2)
\]
\[
C = 2(\sum_j |a(e_j)|)^2
+2 \sum |\Re [a(e_j+e_k)]|.
\]

The following is a standard result.  See Lemma 2.8 of \cite{BB} for instance.

\begin{lemma} \label{Polya2}
If $p\in \C[z_1,\dots, z_n]$ has no zeros in $\C_+^n$ 
then for $z=x+iy\in \C_{+}^n$
\[
|p(x+iy)| \geq |p(x-iy)|.
\]
\end{lemma}

\begin{proof}
The one variable polynomial $q(\zeta) = p(x+\zeta y)$
has no zeros in $\C_+$.  Then, $q$ can be factored as a product of terms of the form $(1+\alpha \zeta)$ where
$\Im \alpha \leq 0$.  We can then check directly that
\[
|1+i \alpha| \geq |1-i\alpha|
\]
which implies $|q(i)|\geq |q(-i)|$.
\end{proof}

\begin{lemma} \label{Polya3}
If $p\in \C[z_1,\dots, z_n]$ has no zeros in $\C_{+}^n$
and if $0\leq y \leq \tilde{y}$ then for any $x\in \R^n$
\[
|p(x+iy)| \leq |p(x+i\tilde{y})|.
\]
\end{lemma}

\begin{proof}
The one variable polynomial $q(\zeta) = p(x+iy+\zeta(\tilde{y}-y))$ has no zeros in $\C_{+}$.  Factors
of $q$ are of the form $(1+\alpha \zeta)$ with $\Im \alpha \leq 0$.  Since $|1+i\alpha| \geq 1$ we have
$|q(0)| \leq |q(i)|$.
\end{proof}

We can get a slightly better bound on $\R^n$ by
modifying the argument of Theorem \ref{msz}.

\begin{theorem}
Suppose $p\in \C[z_1,\dots, z_n]$ is stable.  If $p(0)=1$ then for $x\in \R^n$
\[
\log |p(x)| \leq \Re(\nabla p(0)\cdot x) + \frac{1}{2}(|\nabla p(0)|^2+ \|\Re(Hp)(0)\|) |x|^2
\]
where $Hp$ is the Hessian matrix of $p$.
\end{theorem}

\begin{proof}
We can write $x \in \R^n$ as $x= x_{+} - x_{-}$ where
$(x_+)_j = \begin{cases} x_j & \text{ if } x_j \geq 0 \\ 0 & \text{ if
  } x_j<0 \end{cases}$.  Define
\[
P(z_1,z_2) = p(z_1 x_+ + z_2 x_{-})
\]
which has no zeros in $\C_{+}^2$ and $P(0)=1$.  
Set $S_+ = \{j: x_j \geq 0\}$, $S_{-}=\{j:x_j<0\}$.

Note that 
\[
P_1(z) = \sum_{j\in S_{+}} p_j(z_1x_+ + z_2 x_{-}) |x_j| \qquad P_2(z) = \sum_{j\in S_{-}} p_j(z_1x_+ + z_2 x_{-}) |x_j|
\]
\[
P_{11}(0)= \sum_{j,k\in S_{+}} p_{jk}(0)|x_j||x_k|,  \quad P_{12}(0) =
\sum_{j\in S_{+},k\in S_{-}} p_{jk}(0)|x_j||x_k|, \quad 
P_{22}(0) = \sum_{j,k\in S_{-}} p_{jk}(0) |x_j||x_k|.
\]
Now, since $P(1,-1) = p(x)$ we have
\[
\log|p(x)| \leq \Re( \nabla p(0) \cdot x) 
+ (1/2)(|\nabla p(0)|^2|x|^2 - \Re( \sum_{jk} p_{jk}(0) |x_j||x_k|))
\]
by Theorem \ref{bisz}.

\end{proof}

\section{Multivariable inequalities when $p(0)=0$}\label{sec:msz2}

In this section we prove Theorem \ref{msz2}.
Write the homogeneous 
expansion of $p$
\[
p(z) = \sum_{j=r}^{d} P_j(z).
\]
Notice that $P_r(z)$ is stable itself by Hurwitz's theorem because
\[
P_r(z) = \lim_{t\searrow 0} t^{-r} p(t z)
\]
exhibits $P_r$ as a limit of polynomials with no zeros in $\C_+^{n}$.  

We can make some of the reductions as in the previous section.
We may assume $z= x+iy \in \C_{+}^n$ by Lemma \ref{Polya2}.  
Define $m =\max\{|x_j|,y_j: 1\leq j
\leq n\}$ and $\tilde{y} = m\vec{1}$.
Then, $\tilde{y}\geq \pm x$, $\tilde{y}\geq y$ and $|p(z)| \leq |p(x+i\tilde{y})|$.
Define
\[
q(w_1,w_2) = p(w_1 (\tilde{y} +x) + w_2(\tilde{y}-x))
\]
which is stable and has homogeneous expansion
\[
\sum_{j=r}^{d} Q_j(w) = \sum_{j=r}^{d} P_j(w_1 (\tilde{y} +x) + w_2(\tilde{y}-x)).
\]
All of the terms above are homogeneous of the correct degree but
it is conceivable that the first term vanishes.  Setting $w_1=w_2=1$
we see  the first term evaluates to $P_r(2\tilde{y}) = (2m)^r
P_r(\vec{1})$ which is non-zero.

The data we need for Theorem \ref{bisz2} is:
\[
\begin{aligned}
Q_{j}(\vec{1}) &= (2m)^{j}P_{j}(\vec{1}) \\
\frac{\partial Q_j}{\partial w_1}(\vec{1}) &= (2m)^{j-1}\sum_{k=1}^{n}
\frac{\partial P_j}{\partial z_k}(\vec{1}) (m+x_k) \\
\frac{\partial Q_j}{\partial w_2}(\vec{1}) &= (2m)^{j-1}\sum_{k=1}^{n}
\frac{\partial P_j}{\partial z_k}(\vec{1}) (m-x_k) 
\end{aligned}
\]
and so (omitting some details)
\[
\left|q\left(\frac{i+1}{2},\frac{i-1}{2}\right)\right| \leq (2m)^{r}|P_r(\vec{1})| e^{-r/2}
\exp(\Re( A ) + \frac{1}{2} B )
\]
where
\[
A = \frac{1}{P_r(\vec{1})} \left( \nabla P_r(\vec{1})\cdot
  (x+i\tilde{y})\left(\frac{1}{2m}- \frac{P_{r+1}(\vec{1})}{P_r(\vec{1})}\right)
  + \nabla P_{r+1}(\vec{1})\cdot (x+i\tilde{y})\right)
\]
\[
B = \frac{1}{2}\left( (2m)^2
  \left|\frac{P_{r+1}(\vec{1})}{P_r(\vec{1})}\right|^2 - 2(2m)^2
  \Re\left(\frac{P_{r+2}(\vec{1})}{P_r(\vec{1})}\right) +r\right).
\]
Note $m\leq \|z\|_{\infty}$ and $\|x+i\tilde{y}\|_{\infty} \leq
\sqrt{2} m$.  We can crudely estimate $A$:
\[
\begin{aligned}
|A| &\leq \frac{1}{|P_r(\vec{1})|} \left( \|\nabla
  P_r(\vec{1})\|_{1}(\sqrt{2}m)\left(\frac{1}{2m} +
  \left|\frac{P_{r+1}(\vec{1})}{P_r(\vec{1})}\right|\right)
+ \|\nabla P_{r+1}(\vec{1})\|_{1} \sqrt{2}m\right) \\
&\leq \frac{\|\nabla P_r(\vec{1})\|_1}{\sqrt{2} |P_r(\vec{1})|}
+ \frac{\sqrt{2}}{|P_r(\vec{1})|}\left( \|\nabla P_r(\vec{1})\|_1
    \left|\frac{P_{r+1}(\vec{1})}{P_r(\vec{1})}\right| +
\|\nabla P_{r+1}(\vec{1})\|_{1}\right) \|z\|_{\infty}
\end{aligned}
\]
and
\[
|B| \leq 2\|z\|^2_{\infty}\left(
\left|\frac{P_{r+1}(\vec{1})}{P_r(\vec{1})}\right|^2 - 2
  \Re\left(\frac{P_{r+2}(\vec{1})}{P_r(\vec{1})}\right)\right) + r/2.
\]
Here we use $\|\cdot \|_1$ for $\ell^1$ norm of a vector.
Putting everything together
\[
|p(z)| \leq \|z\|_{\infty}^{r} |P_r(\vec{1})|
\exp(C_0+C_1\|z\|_{\infty} + C_2\|z\|_{\infty}^{2})
\]
where 
\[
C_0 = r(\log(2)-1/4) + \frac{\|\nabla P_r(\vec{1})\|_1}{\sqrt{2} |P_r(\vec{1})|}
\]
\[
C_1 = \frac{\sqrt{2}}{|P_r(\vec{1})|}\left( \|\nabla P_r(\vec{1})\|_1
    \left|\frac{P_{r+1}(\vec{1})}{P_r(\vec{1})}\right| +
\|\nabla P_{r+1}(\vec{1})\|_{1}\right)
\]
\[
C_2 = \left(
\left|\frac{P_{r+1}(\vec{1})}{P_r(\vec{1})}\right|^2 - 2
  \Re\left(\frac{P_{r+2}(\vec{1})}{P_r(\vec{1})}\right)\right).
\]

\end{document}